\documentclass[11pt]{amsart}
\usepackage{geometry}
\usepackage{amsmath}
\usepackage{amsthm}
\usepackage{amssymb}
\usepackage{graphicx}
\usepackage{mathrsfs}
\usepackage{dsfont}
\usepackage{pifont}
\usepackage{ytableau}
\usepackage{bookmark}
\usepackage{tikz}
\usepackage{youngtab}
\usepackage{upgreek}
\hypersetup{hidelinks}
\graphicspath{{./GTM9/}}
\hypersetup{colorlinks,linkcolor=blue,urlcolor=cyan,citecolor=blue}

\newcommand{\la}{\lambda}
\newcommand{\Ga}{\Ga}

\newcommand{\op}{\operatorname}

\newcommand{\mFS}{\mathbb{F}\mathfrak{S}_n}
\newcommand{\F}{\mathbb{F}}
\newcommand{\mhcn}{\mathfrak{H}^{\mathfrak{c}}_n}
\newcommand{\mpcn}{\mathcal{P}^{\mathfrak{c}}_n}
\newcommand{\I}{\mathbb{I}}

\newtheorem{thm}{Theorem}[section]
\theoremstyle{plain}

\newtheorem{lem}[thm]{Lemma}
\newtheorem{prop}[thm]{Proposition}
\newtheorem{cor}[thm]{Corollary}

\theoremstyle{definition}
\newtheorem{defn}[thm]{Definition}
\newtheorem{example}[thm]{Example}

\theoremstyle{remark}
\newtheorem{rem}[thm]{Remark}

\definecolor{A}{rgb}{.75,1,.75}

\numberwithin{equation}{section}

\title[Symmetric groups and Sergeev superalgebras]{A note on irreducible representations of symmetric groups and Sergeev superalgebras}
\author[Minjia Chen, Jinkui Wan and Hongbo Zhao]{Minjia Chen, Jinkui Wan and Hongbo Zhao}

\address{Minjia Chen,  School of Mathematics and Statistics, Beijing Institute of Technology, Beijing 100081, China}
\email{scarleturanus@163.com}
\address{Jinkui Wan, School of Mathematical Sciences, Shenzhen University, Shenzhen, 518060, P.R. China}
\email{wjk302@hotmail.com}
\address{Hongbo Zhao,  School of Mathematics and Statistics, Beijing Institute of Technology, Beijing 100081, China}
\email{2767451726@qq.com}

\begin{document}
\maketitle

\begin{abstract}
We provide an explicit construction and a closed dimension formula in terms of hook lengths for the irreducible 
representations for the symmetric groups $\mathfrak{S}_p$ and the Sergeev 
superalgebras $\mathcal{Y}_p$ over an algebraically closed field $\mathbb{F}$ 
of characteristic $p>0$.
\end{abstract}

\setcounter{tocdepth}{1}
 \tableofcontents

\section{Introduction}

Let $\mathbb{F}$ be an algebraically closed field of characteristic $p > 0$. Let $\mathfrak{S}_n$ denote the symmetric group on $n$ letters, generated by the adjacent transpositions $s_1, s_2, \ldots, s_{n-1}$. We denote by $\mathcal{Y}_n = \mathcal{C}_n \rtimes \mathbb{F}\mathfrak{S}_n$ the associated {Sergeev superalgebra}, where $\mathcal{C}_n$ is the Clifford algebra over $\mathbb{F}$ generated by $c_1, \ldots, c_n$ subject to the relations $c_k^2 = 1$ and $c_k c_l = -c_l c_k$ for $1 \leq k \neq l \leq n$. It is a classical result that both $\mathbb{F}\mathfrak{S}_n$ and $\mathcal{Y}_n$ are semisimple if and only if $n < p$. In this semisimple situation, the irreducible representations are well-characterized, with established seminormal forms and explicit dimension formulas (cf. \cite{K1, WW}).

However, in the {modular case} where $n \geq p$, the representation theory becomes substantially more complex. For over a century, determining a general dimension formula for the irreducible representations of $\mathbb{F}\mathfrak{S}_n$ has remained a major open problem; a similar challenge persists for the Sergeev superalgebra $\mathcal{Y}_n$.

Significant progress has been made by focusing on specific classes of modules. In \cite{Ma}, Mathieu determined the dimensions of irreducible $\mathbb{F}\mathfrak{S}_n$-modules associated with partitions $\lambda=(\lambda_1,\ldots,\lambda_\ell)$ of $n$ with $\ell(\la)=\ell$ satisfying the condition $\lambda_1-\lambda_\ell-\ell \leq p$ by employing classical Schur-Weyl duality. Subsequently, Kleshchev \cite{K2} demonstrated that these representations are precisely the {completely splittable} modules—those whose restrictions to the subgroup $\mathfrak{S}_k$ remain semisimple for all $k < n$, or equivalently, those upon which the Jucys-Murphy elements act semisimply. These results have been generalized to degenerate affine Hecke algebras by Ruff \cite{Ru}, to affine Hecke algebras of type A \cite{Ch, Ra}, and to Khovanov-Lauda-Rouquier algebras \cite{KR}. Furthermore, the second author extended these results to degenerate affine Hecke-Clifford algebras; as a byproduct, an explicit construction and dimension formula for the completely splittable representations of $\mathcal{Y}_n$ (for $p \neq 2$) were obtained using special standard Young tableaux in \cite{Wa}.

The present note provides a closed dimension formula and an explicit construction for the irreducible representations of $\mathbb{F}\mathfrak{S}_n$ and $\mathcal{Y}_n$ specifically in the case $n=p$. Our approach is based on the observation that when $n \leq p$, {all} irreducible representations of $\mathbb{F}\mathfrak{S}_n$ and $\mathcal{Y}_n$ are completely splittable. More specifically, by comparing the parametrization set of irreducible modules with that of completely splittable representations, we observe that these two sets coincide if and only if $n \leq p$. 

Under the assumption $n \leq p$, we prove that the combinatorial dimension formulas for these irreducible modules can be expressed explicitly in terms of the {hook lengths} associated to the corresponding Young diagrams. Notably, in the non-semsimple case where $n=p$, we obtain a closed dimension formula for all irreducible representations of $\mathfrak{S}_p$ (resp. $\mathcal{Y}_p$) and further give the explicit action of the generators $s_1,s_2,\ldots,s_{p-1}$ of $\mathfrak{S}_p$ (resp. generators $c_1,\ldots,c_p,s_1,\ldots,s_{p-1}$ of  $\mathcal{Y}_p$) on the underlying vector spaces. Inspired by this work, in the forthcoming article \cite{CW}, we develop an analogous construction for Hecke-Clifford algebra $\mathcal{H}_n(q)$ which is the $q$-analog of $\mathcal{Y}_n$ and as an application we then further establish a semi-simplicity criterion for $\mathcal{H}_n(q)$ at roots of unity.

The paper is organized as follows. In Section \ref{sec:symmetric}, we recall preliminary results regarding the symmetric group and derive a closed dimension formula for irreducible modules in terms of hook lengths, alongside a construction of these modules via the detailed action of the generators $s_1, \ldots, s_{p-1}$ in the case $n= p$. In Section \ref{sec:Sergeev}, we provide an analogous dimension formula and explicit construction for the irreducible representations of the Sergeev superalgebra.

\section{Irreducible representations of symmetric group $\mathfrak{S}_p$}\label{sec:symmetric}
In this section, we will derive a closed dimension formula and an explicit construction for irreducible representations of $\mathbb{F}\mathfrak{S}_p$ over an algebraically closed field $\mathbb{F}$ of characteristic $p>0$. 

\subsection{Basics on partitions and Specht modules} 
Let $\mathcal{P}(n)$ denote the set of partitions of $n$. For a partition $\la=(\la_1,\la_2,\dots)\in \mathcal{P}(n)$, we always assume $\la_i\geq\la_{i+1}\geq 0$ and $\sum_i\la_i=n$. Set $\ell(\la)=\sharp\{i\mid \la_i>0\}$. We can also write a partition as $\la=(1^{a_1}2^{a_2}\cdots)$ with $a_i$ being the number of  parts of $\la$ equal to $i$ for $i\geq 1$. 
It is known that any partition $\la \in \mathcal{P}(n)$ can be identified with its Young diagram, that is $\la=\{(i,j)\in\mathbb{Z}^2 \mid 1\leq i\leq \ell(\la), 1\leq j\leq \la_i\}$. 
For $\la \in \mathcal{P}(n)$, the conjugate partition $\la^\prime$ is defined by $\la^\prime_j=\op{max}\{i:\la_i\geq j\}$. The $(i,j)$-hook of $\la$ is the set of nodes in the Young diagram that are either in the same row as $(i,j)$ and to the right of $(i,j)$, or in the same column as $(i,j)$ and below $(i,j)$, including $(i,j)$ itself. The hook length of the $(i,j)$-hook is defined as $h_\la(i,j)=\la_i+\la^\prime_j-i-j+1$, which equivalently counts the number of nodes in the $(i,j)$-hook. 
Denote by $\mathcal{T}(\la)$ the set of tableaux of shape $\la$; that is, a tableau is a labelling of the nodes in the Young diagram $\la$ with the entries $1,2,\dots, n$. A tableau $T$ is called \emph{standard} if its entries strictly increase from left to right along each row and down each column. We denote by $\op{Std}(\la)$ the subset of $\mathcal{T}(\la)$ consisting of standard tableaux of shape $\la$. We then have the following remarkable hook length formula (cf. \cite[Chapter I, Section 5, Example 2]{Mac}):
\begin{equation}\label{eq:hook}
f^\la:=\sharp \op{Std}(\la)=\frac{n!}{\prod_{(i,j)\in\la} h_\la(i,j)}
\end{equation}
where the product in the denominator is over all nodes in the Young diagram $\la$.

\subsection{Irreducible  $\mathbb{F}\mathfrak{S}_n$-modules in the case $n=p$}

Let $\mathcal{P}_p(n)$ be the set of $p$-regular partition of $n$, that is, the subset of $\mathcal{P}(n)$ consisting of $\la$ such that no part of $\la$ is repeated more than or equal to $p$ times. The element in $\mathcal{P}_p(n)$ can also be written in the form
\[
\mathcal{P}_p(n)=\{\la=(1^{a_1}2^{a_2}\cdots)\in\mathcal{P}(n)\mid 0\leq a_k<p \text{ for all } k\}. 
\]
It is known that for each $\la \in \mathcal{P}(n)$, there exists a {\em Specht module} $S^\la$ of dimension $f^\la$, which further admits a symmetric $\mathfrak{S}_n$-invariant inner product $(~,~)$. The radical of this inner product, denoted by $\text{Rad}^\la$, is a $\mathfrak{S}_n$-submodule of $S^\la$. We refer the reader to \cite{Ja} for details regarding the construction of $S^\la$ and the inner product $(~,~)$, which we omit here as they are not required for our current purposes. Set
$$
D^\la=S^\la/\text{Rad}^\la. 
$$
\begin{thm}\cite{Ja} For $\la\in\mathcal{P}(n)$, $D^\la\neq 0$ if and only if $\la\in\mathcal{P}_p(n)$. Moreover, the set $\{D^\la\mid \la\in\mathcal{P}_p(n)\}$ is a complete set of non-isomorphic irreducible $\mFS$-modules.
\end{thm}

Define the Jucys-Murphy elements in $\mFS$ as follows:
\[L_1=0,\quad L_k=\sum_{1\leq m<k}(m,k),~2\leq k\leq n.\]
An irreducible $\mathbb{F}\mathfrak{S}_n$-module is called \emph{completely splittable} if the Jucys-Murphy elements $L_1,L_2,\ldots,L_n$ act semisimply.  For $\la=(\la_1,\la_2,\cdots,\la_\ell)\in\mathcal{P}_p(n)$ of $n$ with $\ell(\la)=\ell$, let
\[\chi(\la)=\la_1-\la_\ell+\ell.\]
Then $\chi(\la)$ coincides with a special hook length, that is,  $\chi(\la)=h_\la(1,\la_r)$. 
\begin{example}
Let $\la=(8,7,7,5,3,3)$. Then $\chi(\la)=11$ and it corresponds to the $(1,3)$-hook as follows
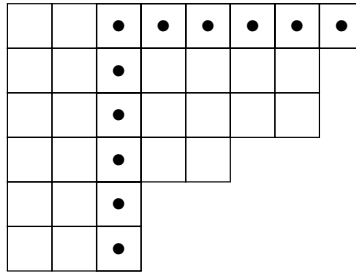
\[\ydiagram{8,7,7,5,3,3}*[\bullet]{2+6,2+1,2+1,2+1,2+1,2+1}\]
\end{example}

For $T\in\op{Std}(\la)$ with $\la\in\mathcal{P}(n)$, let $T_{(i,j)}$ denote the entry in the node $(i,j)\in\la$. 
Let 
$$
\mathcal{CP}_p(n)=\{\la\in \mathcal{P}_p(n)|\chi(\la)\leq p\}.
$$
For $\la \in \mathcal{CP}_p(n)$,  a standard tableau $T\in{\rm Std}(\la)$ is called \emph{$p$-standard} if any two of its entries $T_{(i,j)}$ and $T_{(i',j')}$ in the node $(i,j)$ and $(i',j')$ with $i>i^\prime,j<j^\prime$ and $i+j^\prime+1-i^\prime-j=p$ satisfy $T_{(i^\prime,j^\prime)}>T_{(i,j)}$. Denote by $\op{Std}_p(\la)$ the set of $p$-standard tableaux of shape $\la$ for each $\la\in\mathcal{CP}_p(n)$, that is 
\begin{align}
\op{Std}_p(\la)=\Big\{T\in\op{Std}(\la)\mid &T_{(i^\prime,j^\prime)}>T_{(i,j)} \text{ for any }(i,j),(i',j')\in\la\text{ with }\notag\\
&\quad i>i^\prime,j<j^\prime, i+j^\prime+1-i^\prime-j=p\Big\}. \label{eq:Stdp}
\end{align}
\begin{example}
Take $p=5,\la=(4,3,1)$. Then the tableau
\[T=\ytableaushort{1256,378,4}\]
is $p$-standard. But  the tableau 
\[s_4T=\ytableaushort{1246,378,5}\]
is standard but not $p$-standard. 
\end{example}

\begin{thm}\cite[Theorem 2.1 and Corollary 2.4]{K2}\label{thm:K2}
Let $\la\in\mathcal{P}_p(n)$. Then $D^\la$ is completely splittable if and only if $\la \in \mathcal{CP}_p(n)$. Moreover, if $\la \in \mathcal{CP}_p(n)$, then
\[\op{dim} D^\la=\sharp \op{Std}_p(\la).\]
\end{thm}

Suppose $\la\in\mathcal{CP}_p(n)$. For a node $A=(i,j)$ in a Young diagram $\la$, define the residue of $A$ to be ${\rm res}(A)=j-i \mod p$. For $1\leq k\leq n$ and $T\in{\rm Std}_p(\la)$, set $T(k)$ be the node in $T$ occupied by $k$.  It is straightforward to check that $\op{res}(T(a+1))\neq \op{res}(T(a))$ if  $T\in{\rm Std}_p(\la)$ and hence 
$$
\rho_a(T):=\frac{1}{\op{res}(T(a+1))-\op{res}(T(a))} 
$$
is well defined for each $1\leq a\leq n-1$. 

\begin{thm}\cite[Theorem 4.9]{Ru}
Suppose $\la\in\mathcal{CP}_p(n)$. There exists an $\mathbb{F}\mathfrak{S}_n$-module isomorphism 
$$D^\la\cong \oplus_{T\in\op{Std}_p(\la)}\mathbb{F}v_T \text{ as a vector space}, $$
where the action of simple transposition $s_i\in \mathfrak{S}_n$ for $1\leq i\leq n-1$ is given as follows:
\begin{equation}\label{eq:sk-action}
s_i v_T=\left\{
\begin{array}{ll}
\rho_i(T)v_T+\sqrt{1-\left(\rho_i(T)\right)^2}v_{s_i T},&\text{ if }s_iT \in{\rm Std}_p(\la),\\
\rho_{i}(T)v_T,&\text{ otherwise}
\end{array}
\right.
\end{equation}
for each $T\in{\rm Std}_p(\la)$.  
\end{thm}

\begin{lem}\label{lem:equal}
$\mathcal{CP}_p(n)=\mathcal{P}_p(n)$ if and only if $n\leq p$. Hence, all irreducible $\mathbb{F}\mathfrak{S}_n$-modules are completely splittable if and only if $n\leq p$.
\end{lem}
\begin{proof}
Suppose $n\leq p$. For any $p$-regular partition $\la=(\la_1,\la_2,\cdots,\la_\ell)$ of $n$ with $\ell(\la)=\ell$, we have $\chi(\la)=\la_1-\la_\ell+\ell=n-(\la_2+\la_3+\cdots+\la_\ell)-\la_\ell+\ell\leq n\leq p$. Hence $\mathcal{P}_p(n)\subseteq\mathcal{CP}_p(n)$ which leads to $\mathcal{CP}_p(n)=\mathcal{P}_p(n)$.  Conversely, 
if $n> p$, then obviously $\chi((n-1,1))=n>p$ while $(n-1,1)\in\mathcal{P}_n(n)$. This means $\mathcal{CP}_p(n)\subsetneq\mathcal{P}_p(n)$ and thus $\mathcal{CP}_p(n)\neq\mathcal{P}_p(n)$ if $n>p$.    
This proves the lemma. 
\end{proof}

\begin{prop}\label{prop:p-standard}
Suppose $n\leq p$ and $\la\in\mathcal{CP}_p(n)$, then we have
\[\sharp \op{Std}_p(\la)=\begin{cases}
f^{(k-1,1^{n-k})}=\binom{n-2}{k-2}, &\text{ if }\la=(k,1^{n-k})\text{ for some }2\leq k\leq n-1,\\
f^\la=\frac{n!}{\prod_{(i,j)\in\la}h_\la(i,j)}, &\text{ otherwise}. 
\end{cases}\]
\end{prop}
\begin{proof}
Clearly in the case $n<p$ we have $\mathcal{CP}_p(n)=\mathcal{P}_p(n)=\mathcal{P}(n)$ and moreover $\op{Std}_p(\la)=\op{Std}(\la)$ by \eqref{eq:Stdp}. Thus the proposition holds by  \eqref{eq:hook}. 
Now assume $n=p$.  It is direct to verify 
$$
\mathcal{P}_p(n)=\mathcal{P}(n)\setminus \{(1^n)\}.
$$ 
and hence $\mathcal{CP}_p(n)=\{\la\in\mathcal{P}(n)\mid \la\neq (1^n)\}$ by Lemma \ref{lem:equal}. 
Assume $\la\in\mathcal{P}(n)$ with $\la\neq (1^n)$. Notice that for any two nodes $(i,j)$ and $(i^\prime,j^\prime)$ in the Young diagram $\la$ with $i>i^\prime,j<j^\prime$, we have 
\begin{equation}\label{eq:p-standard}
i+j^\prime+1-i^\prime-j\leq \la^\prime_j+\la_{i^\prime}+1-i^\prime-j=h_{\la}(i^\prime,j)
\end{equation}
and moreover the equality in \eqref{eq:p-standard} holds if and only if $i=\la^\prime_j,j^\prime=\la_{i^\prime}$. Meanwhile if $\la_2\geq 2$ then 
\begin{equation}\label{eq:hook-inequality}
h_\la(i,j)<n=p 
\end{equation}
for any node $(i,j)$ in the Young diagram $\la$. 
Then by \eqref{eq:p-standard} and \eqref{eq:hook-inequality} we obtain that every $T\in\op{Std}(\la)$ is $p$-standard if $\la_2\geq 2$. Thus 
\begin{equation}\label{eq:Stdp-equal-1}
\op{Std}_p(\la)=\op{Std}(\la)\text{ if } \la\neq (k,\underbrace{1,1,\ldots,1}_{n-k})=(k,1^{n-k})
\end{equation}
Now suppose $\la_2\leq 1$. Then $\la$ must be of the form $ \la=(k,1^{n-k})$ with $k\geq 2$. Suppose $(i,j)$ and $(i',j')$ are two nodes in the Young diagram $\la$ satisfying $i>i^\prime,j<j^\prime$ and $i+j^\prime+1-i^\prime-j=p$. Then clearly we have  $(i',j')=(1,k)$ and  $(i,j)=(n-k+1,1)$. Hence a tableaux $T\in \op{Std}(\la)$ belongs to $\op{Std}_p(\la)$ if and only if $T_{(1,k)}>T_{(n-k+1,1)}$ which means $T_{(1,k)}=n$. This implies that 
$
\op{Std}_p(\la)=\{T\in \op{Std}(\la)\mid T_{(1,k)}=n\}
$
and thus 
\begin{equation}\label{eq:Stdp-equal-2}
\sharp\op{Std}_p(\la)=\sharp\op{Std}(\la^-), \quad \text{ where }\la^-=(k-1,1^{n-k}). 
\end{equation}
Then the proposition follows from \eqref{eq:Stdp-equal-1}, \eqref{eq:Stdp-equal-2} and \eqref{eq:hook}. 
\end{proof}
Then we are ready to introduce one of our main results.
\begin{thm}
Suppose $n=p$ and $\la\in\mathcal{CP}_p(n)$. 
\\
(1) If $ \la\neq (k,1^{n-k})$, then there exists an isomorphism of $\mFS$-modules 
\begin{equation*}
{D}^\la\cong \oplus_{T\in{\rm Std}(\la)}\mathbb{F}v_T \text{ as a vector space over }\mathbb{F}, 
\end{equation*}
where the action $s_i\in\mathfrak{S}_n$ for $1\leq i\leq n-1$ satisfies 
\begin{equation}\label{eq:sk-action-2}
s_i v_T=\left\{
\begin{array}{ll}
\rho_i(T)v_T+\sqrt{1-\left(\rho_i(T)\right)^2}v_{s_i T},&\text{ if }s_iT \in{\rm Std}(\la),\\
\rho_i(T)v_T,&\text{ otherwise.}
\end{array}
\right.
\end{equation}
In addition, $\dim D^\la=\frac{n!}{\prod_{(i,j)\in\la}h_\la(i,j)}.$
\\
(2) If $\la=(k,1^{n-k})$ with $k\geq 2$, there exists an isomorphism $\mFS$-modules 
 \begin{equation*}
D^\la\cong\oplus_{T\in{\rm Std}(\la^-)}\mathbb{F}v_T \text{ as a vector space over }\mathbb{F}, 
\end{equation*}
with $\la^-=(k-1,1^{n-k})$ and the action $s_i\in\mathfrak{S}_n$ for $1\leq i\leq n-1$ satisfies 
\begin{equation}\label{eq:sk-action-3}
\begin{aligned}
s_i v_T=&\left\{
\begin{array}{ll}
\rho_i(T)v_T+\sqrt{1-\left(\rho_i(T)\right)^2}v_{s_i T}&\text{ if } 1\leq i\leq n-2\text{ and }s_iT \in{\rm Std}(\la^-) ,\\
\rho_i(T)v_T,&\text{ if }1\leq i\leq n-2\text{ and }s_iT \notin{\rm Std}(\la^-) 
\end{array}
\right.\\
s_{n-1}v_T=&\left\{
\begin{array}{cc}
v_T,&\text{ if } n-1=T_{(1,k-1)},\\
-v_T,&\text{ if }n-1=T_{(n-k+1,1)}. 
\end{array}
\right.\
\end{aligned}
\end{equation}
In addition, $\dim D^\la=\binom{n-2}{k-2}.$
\\
(3) The set $\{D^\la\mid \la\in\mathcal{P}(n), \la\neq (1^n)\}$ is a complete set of non-isomorphic irreducible $\mFS$-modules. 
\end{thm}
\begin{proof}
The first part (1) clearly is due to \eqref{eq:sk-action}, Theorem \ref{thm:K2} and Proposition \ref{prop:p-standard}. It remains to prove the second part (2). Now assume $\la=(k,1^{n-k})$ with $k\geq 2$. 
Set $\la^-=(k-1,1^{n-k})$. Then by the proof of Proposition \ref{prop:p-standard}, we obtain that there exists a bijection 
\begin{equation}
\phi: {\rm Std}_p(\la)\longrightarrow {\rm Std}(\la^-),\qquad  T\mapsto \phi(T)=T\setminus (1,k)
\end{equation}
Clearly $\phi(s_iT)=s_i\phi(T)$ for $1\leq i\leq n-2$. Moreover $s_{n-1}T\notin {\rm Std}_p(\la)$. Then one can identify $v_T$ with $v_{\phi(T)}$. Again (2) follows from \eqref{eq:sk-action}, Theorem \ref{thm:K2} and Proposition \ref{prop:p-standard}. 
This proves the theorem. 
\end{proof}

\begin{cor}
Suppose $n=p$ and $\la\neq (1^n)$. 

(1) If $\la\neq (k,1^{n-k})$, then ${D}^\la\cong S^\la$ and hence ${\rm Rad}^\la=0$. 

(2) If $\la=(k,1^{n-k})$ with $k\geq 2$, then $\dim {\rm Rad}^\la=\binom{n-2}{k-1}$.  
\end{cor}

\section{Irreducible representations of Sergeev superalgebra $\mathcal{Y}_p$}\label{sec:Sergeev}
In this section we assume that $\F$ is an algebraically closed field of characteristic
$p$ with $p\neq 2$ or equivalently $p\geq 3$. We will derive a closed dimension formula and an explicit construction for irreducible representations of the Sergeev superalgebra $\mathcal{Y}_p$ over  $\mathbb{F}$. 

\subsection{Basics on affine Sergeev superalgebras}
{\color{black}
We first recall some basic on superalgebras, referring the
reader to~\cite[Chapter 12]{K1}. A vector superspace $V$ means a $\mathbb{Z}_2$-graded space $V=V_{\bar{0}}\oplus V_{\bar{1}}$ over $\mathbb{F}$. Denote by 
$\bar{v}\in\mathbb{Z}_2$ the parity of a homogeneous vector $v$ of a
vector superspace. By a superalgebra, we mean a
$\mathbb{Z}_2$-graded associative algebra. Let $\mathcal{A}$ be a
superalgebra. An $\mathcal{A}$-module means a $\mathbb{Z}_2$-graded
left $\mathcal{A}$-module and a homomorphism $f:V\rightarrow W$ of
$\mathcal{A}$-modules $V$ and $W$ means a linear map such that $
f(av)=(-1)^{\bar{f}\bar{a}}af(v).$  Note that this and other such
expressions only make sense for homogeneous $a, f$ and the meaning
for arbitrary elements is to be obtained by extending linearly from
the homogeneous case.  Let $V$ be a finite dimensional
$\mathcal{A}$-module. Let $\Pi
 V$ be the same underlying vector space but with the opposite
 $\mathbb{Z}_2$-grading. The new action of $a\in\mathcal{A}$ on $v\in\Pi
 V$ is defined in terms of the old action by $a\cdot
 v:=(-1)^{\bar{a}}av$. Note that the identity map on $V$ defines
 an isomorphism from $V$ to $\Pi V$.

By a superalgebra analog of Schur's Lemma,  the endomorphism
algebra of a finite dimensional irreducible module over a
superalgebra is either one dimensional or two dimensional. In the
former case, we call the module of {\em type }\texttt{M} while in
the latter case the module is called of {\em type }\texttt{Q}.

Given two superalgebras $\mathcal{A}$ and $\mathcal{B}$, 
the tensor product of superspaces $\mathcal{A}\otimes\mathcal{B}$
can be viewed as a superalgebra with multiplication defined by
$$
(a\otimes b)(a'\otimes b')=(-1)^{\bar{b}\bar{a'}}(aa')\otimes (bb')
\qquad (a,a'\in\mathcal{A}, b,b'\in\mathcal{B}).
$$
Suppose $V$ is an $\mathcal{A}$-module and $W$ is a
$\mathcal{B}$-module. Then $V\otimes W$ affords an $A\otimes B$-module
denoted by $V\boxtimes W$ via
$$
(a\otimes b)(v\otimes w)=(-1)^{\bar{b}\bar{v}}av\otimes bw,~a\in A,
b\in B, v\in V, w\in W.
$$

 If $V$ is an irreducible $\mathcal{A}$-module and $W$ is an
irreducible $\mathcal{B}$-module, $V\boxtimes W$ may not be
irreducible. Indeed, we have the following standard lemma (cf.
\cite[Lemma 12.2.13]{K1}).
\begin{lem}\label{tensorsmod}
Let $V$ be an irreducible $\mathcal{A}$-module and $W$ be an
irreducible $\mathcal{B}$-module.
\begin{enumerate}
\item If both $V$ and $W$ are of type $\texttt{M}$, then
$V\boxtimes W$ is an irreducible
$\mathcal{A}\otimes\mathcal{B}$-module of type $\texttt{M}$.

\item If one of $V$ or $W$ is of type $\texttt{M}$ and the other
is of type $\texttt{Q}$, then $V\boxtimes W$ is an irreducible
$\mathcal{A}\otimes\mathcal{B}$-module of type $\texttt{Q}$.

\item If both $V$ and $W$ are of type $\texttt{Q}$, then
$V\boxtimes W\cong X\oplus \Pi X$ for a type $\texttt{M}$
irreducible $\mathcal{A}\otimes\mathcal{B}$-module $X$.
\end{enumerate}
Moreover, all irreducible $\mathcal{A}\otimes\mathcal{B}$-modules
arise as constituents of $V\boxtimes W$ for some choice of
irreducibles $V,W$.
\end{lem}
If $V$ is an irreducible $\mathcal{A}$-module and $W$ is an
irreducible $\mathcal{B}$-module, denote by $V\circledast W$ an
irreducible component of $V\boxtimes W$. Thus,
$$
V\boxtimes W=\left\{
\begin{array}{ll}
V\circledast W\oplus \Pi (V\circledast W), & \text{ if both } V \text{ and } W
 \text{ are of type }\texttt{Q}, \\
V\circledast W, &\text{ otherwise }.
\end{array}
\right.
$$

\begin{defn}For $n\geq 1$, the affine Sergeev  superalgebra $\mhcn$  is
the superalgebra generated by even generators
$s_1,\ldots,s_{n-1},x_1,\ldots,x_n$ and odd generators
$c_1,\ldots,c_n$ subject to the following relations
\begin{align}
s_i^2=1,\quad s_is_j =s_js_i, \quad
s_is_{i+1}s_i&=s_{i+1}s_is_{i+1}, \quad|i-j|>1,\notag\\
x_ix_j&=x_jx_i, \quad 1\leq i,j\leq n, \label{poly}\\
c_i^2=1,c_ic_j&=-c_jc_i, \quad 1\leq i\neq j\leq n, \label{clifford}\\
s_ix_i&=x_{i+1}s_i-(1+c_ic_{i+1}),\notag\\
s_ix_j&=x_js_i, \quad j\neq i, i+1, \notag\\
s_ic_i=c_{i+1}s_i, s_ic_{i+1}&=c_is_i,s_ic_j=c_js_i,\quad j\neq i, i+1, \notag\\
x_ic_i=-c_ix_i, x_ic_j&=c_jx_i,\quad 1\leq i\neq j\leq n. \label{xc}
\end{align}
\end{defn}
\begin{rem} The affine Sergeev superalgebra $\mhcn$ was introduced by
Nazarov~\cite{Na}(called affine Sergeev algebra) to study the
representations of $\mathbb{C}\mathfrak{S}_n^-$.  The quantized version of the
$\mhcn$ introduced later by Jones-Nazarov~\cite{JN} to study the
$q$-analogues of Young symmetrizers for projective representations
of the symmetric group $\mathfrak{S}_n$ is often also called affine
Hecke-Clifford algebras.
\end{rem}

For $\alpha=(\alpha_1,\ldots,\alpha_n)\in\mathbb{Z}_+^n$ and
$\beta=(\beta_1,\ldots,\beta_n)\in\mathbb{Z}_2^n$, set
$x^{\alpha}=x_1^{\alpha_1}\cdots x_n^{\alpha_n}$ and
$c^{\beta}=c_1^{\beta_1}\cdots c_n^{\beta_n}$. Then we have the
following.
\begin{lem}\cite[Theorem 2.2]{BK}\label{lem:PBW}
The set $\{x^{\alpha}c^{\beta}w~|~ \alpha\in\mathbb{Z}_+^n,
\beta\in\mathbb{Z}_2^n, w\in \mathfrak{S}_n\}$ forms a basis of $\mhcn$.
\end{lem}
%

%

For each
$i\in\mathbb{Z}$, set
\begin{align}
\mathtt{q}(i)=i(i+1).\label{qi}
\end{align}
Denote by $\mathbb{Z}_+$ the set of nonnegative
integers and let
\begin{align}
\mathbb{I}=\left\{
\begin{array}{ll}
\mathbb{Z}_+, & \text{ if } p=0, \\
\{0,1,\ldots,\frac{p-1}{2}\}, &\text{ if } p\geq 3.
\end{array}
\right.
\label{defn:I} 
\end{align}
Then it is easy to verify 
\begin{equation}\label{eq:qi=qj}
\text{ if }i,j\in\mathbb{I}, \text{ then }\mathtt{q}(i)=\mathtt{q}(j) \text{ if and only if }i=j.
\end{equation} and moreover $\{\mathtt{q}(i)\mid i\in\mathbb{Z}\}=\{\mathtt{q}(i)\mid i\in\mathbb{I}\}$.  This justifies the introduction of $\I$. 
Denote by $\mpcn$ the superalgebra generated by even generators
$x_1,\ldots,x_n$ and odd generators $c_1,\ldots,c_n$ subject to the
relations~(\ref{poly}),~(\ref{clifford}) and~(\ref{xc}). By
Lemma~\ref{lem:PBW}, $\mpcn$ can be identified with the subalgebra
of $\mhcn$ generated by $x_1,\ldots,x_n$ and $c_1,\ldots,c_n$. For a
composition $\mu=(\mu_1,\mu_2,\ldots,\mu_r)$ of $n$, we define
$\mathfrak{H}_{\mu}^{\mathfrak{c}}$ to be the subalgebra of $\mpcn$
generated by $\mpcn$ and $s_j\in \mathfrak{S}_{\mu}=\mathfrak{S}_{\mu_1}\times\cdots
\times \mathfrak{S}_{\mu_r}$. Note that
$\mpcn=\mathfrak{H}_{(1^n)}^{\mathfrak{c}}$. Let us denote by
$\operatorname{Rep}_{\mathbb{I}}\mathfrak{H}_{\mu}^{\mathfrak{c}}$ the
category of so-called \emph{integral} finite dimensional
$\mathfrak{H}_{\mu}^{\mathfrak{c}}$-modules on which the
$x^2_1,\ldots, x^2_n$ have eigenvalues of the form $\mathtt{q}(i)$ for
$i\in\mathbb{I}$. For each $i\in\mathbb{I}$, denote by $L(i)$ the $2$-dimensional
$\mathcal{P}_1^{\mathfrak{c}}$-module with
$L(i)_{\bar{0}}=\mathbb{F}v_0$ and $L(i)_{\bar{1}}=\mathbb{F}v_1$
and
\begin{equation}\label{eq:P1-irrep}
x_1v_0=\sqrt{\mathtt{q}(i)}v_0,\quad x_1v_1=-\sqrt{\mathtt{q}(i)}v_1, \quad
c_1v_0=v_1,\quad c_1v_1=v_0.
\end{equation}
Note that $L(i)$ is irreducible of type $\texttt{M}$ if $i\neq 0$,
and irreducible of type $\texttt{Q}$ if $i=0$. Moreover $\{L(i)\mid 
i\in\mathbb{I}\}$ forms a complete set of pairwise non-isomorphic irreducible
$\mathcal{P}_1^{\mathfrak{c}}$-module in the category
$\operatorname{Rep}_{\mathbb{I}}\mathcal{P}_1^{\mathfrak{c}}$. Observe that
$\mathcal{P}^{\mathfrak{c}}_n\cong \mathcal{P}_1^{\mathfrak{c}}\otimes\cdots\otimes
\mathcal{P}_1^{\mathfrak{c}}$, and hence we have the following
result by Lemma~\ref{tensorsmod}.

\begin{lem} \cite[Lemma 4.8]{BK}\label{lem:irrepPn}
The set of $\mpcn$-modules
$$
\{L(\underline{i})=L(i_1)\circledast
L(i_2)\circledast\cdots\circledast
L(i_n)|~\underline{i}=(i_1,\ldots,i_n)\in\mathbb{I}^n\}
$$
forms a complete set of pairwise non-isomorphic irreducible
$\mpcn$-module in the category $\operatorname{Rep}_{\mathbb{I}}\mpcn$.
Moreover, denote by $\gamma_0$ the number of $1\leq j\leq n$ with
$i_j=0$. Then $L(\underline{i})$ is of type $\texttt{M}$ if
$\gamma_0$ is even and type $\texttt{Q}$ if $\gamma_0$ is odd.
Furthermore,
$\text{dim}~L(\underline{i})=2^{n-\lfloor\frac{\gamma_0}{2}\rfloor}$,
where $\lfloor\frac{\gamma_0}{2}\rfloor$ denotes the greatest
integer less than or equal to $\frac{\gamma_0}{2}$ .
\end{lem}

\begin{rem}\label{rem:Ltau}
Note that each permutation $\tau\in \mathfrak{S}_n$ defines a superalgebra
isomorphism $\tau:\mpcn\rightarrow \mpcn$ by mapping $x_k$ to
$x_{\tau(k)}$ and $c_k$ to  $c_{\tau(k)}$, for $1\leq k\leq n$. For
$\underline{i}\in\mathbb{I}^n$, the twist of the action of
$\mathcal{P}_n^{\mathfrak{c}}$ on $L(\underline{i})$ with
$\tau^{-1}$ leads to a new $\mpcn$-module denoted by
$L(\underline{i})^{\tau}$ with
$$
L(\underline{i})^{\tau}=\{z^{\tau}~|~z\in L(\underline{i})\} ,\quad
fz^{\tau}=(\tau^{-1}(f)z)^{\tau}, \text{ for any }f\in
\mathcal{P}_n^{\mathfrak{c}}, z\in L(\underline{i}).
$$
So in particular we have $(x_kz)^{\tau}=x_{\tau(k)}z^{\tau}$ and
$(c_kz)^{\tau}=c_{\tau(k)}z^{\tau}$. It is easy to see that 
\begin{equation}\label{eq:Ltau}
L(\underline{i})^{\tau}\cong L(\tau\cdot \underline{i}), 
\end{equation}
where
$\tau\cdot \underline{i}:=(i_{\tau^{-1}(1)},\ldots,i_{\tau^{-1}(n)})$
for $\underline{i}=(i_1,\ldots,i_n)\in\mathbb{I}^n$ and $\tau\in \mathfrak{S}_n$.
\end{rem}

\begin{defn}\cite[Definition 3.2]{Wa}
A finite dimensional $\mhcn$-modules is said to be completely splittable if the elements $x_1,x_2,\ldots,x_n$ act semisimply. 
\end{defn}

It is known that a classification of irreducible completely splittable $\mhcn$-module has been obtained in \cite[Theorem 4.5, Theorem 5.7, Theorem 5.15 ]{Wa}. We will give a brief review in the following. 

\begin{defn}\label{defn:cs-wt}
Let $\mathfrak{P}(\mhcn)\subset\I^n$ be the subset consisting of $\underline{i}=(i_1,i_2,\ldots,i_n)\in\I^n$ satisfying the following conditions: 
\begin{enumerate}
\item $i_k\neq i_{k+1}$ for all $1\leq k\leq n-1$.

\item The element $\frac{p-1}{2}\in\I$ appears at most once in
$\underline{i}$.

\item If $i_k=i_l=0$ for some $1\leq k<l\leq n$, then
$1\in\{i_{k+1},\ldots, i_{l-1}\}$.

\item If $i_k=i_l\geq 1$ for
some $1\leq k<l\leq n$, then either of the following holds:
\begin{enumerate}
\item $\{i_k-1,i_k+1\}\subseteq\{i_{k+1},\ldots, i_{l-1}\}$,

\item there exists a sequence of integers
$k\leq r_0< r_1<\cdots< r_{\frac{p-3}{2}-i_k}< q<
t_{\frac{p-3}{2}-i_k}<\cdots< t_1<t_0\leq l$ such that
$i_q=\frac{p-1}{2}, i_{r_j}=i_{t_j}=i_k+j$ and $i_k+j$ does not
appear between $i_{r_j}$ and $i_{t_j}$ in $\underline{i}$ for each
$0\leq j\leq \frac{p-3}{2}-i_k$.
\end{enumerate}
\end{enumerate}
\end{defn}

For $\underline{i}\in\I^n$ and $1\leq k\leq n-1$,  the 
simple transposition $s_k$ is said to be admissible with respect
to $\underline{i}$ if $i_k\neq i_{k+1}\pm1$. Define an equivalence
relation $\sim$ on $\I^n$ by declaring that
$\underline{i}\sim\underline{j}$ if there exist $s_{k_1},\ldots,
s_{k_t}$ for some $t\in\mathbb{Z}_+$ such that
$\underline{j}=(s_{k_t}\cdots s_{k_1})\cdot\underline{i}$ and
$s_{k_{l}}$ is admissible with respect to $(s_{k_{l-1}}\cdots
s_{k_1})\cdot\underline{i}$ for $1\leq l\leq t$.  
Observe that if $\underline{i}\in \mathfrak{P}(\mhcn)$ and $s_k$ is
admissible with respect to $\underline{i}$,  then the conditions in
Definition~\ref{defn:cs-wt} hold for $s_k\cdot\underline{i}$ and
hence $s_k\cdot\underline{i}\in\mathfrak{P}(\mhcn)$. This means there is an
equivalence relation denoted by $\sim$ on $\mathfrak{P}(\mhcn)$ inherited from
the equivalence relation $\sim$ on $\I^n$.
For each
$\underline{i}\in \mathfrak{P}(\mhcn)$, set
\begin{align*}
\Lambda_{\underline{i}}&=\{\underline{j}\in\mathfrak{P}(\mhcn)\mid \underline{j}\sim \underline{i}\}, \\
P_{\underline{i}}&=\{\tau=s_{k_t}\cdots s_{k_1}|~s_{k_a}\text{ is
admissible with respect to } s_{k_{a-1}}\cdots
s_{k_1}\cdot\underline{i}, 1\leq a\leq t,
t\in\mathbb{Z}_+\}.
\end{align*}
Then by \cite{Wa} we have the following. 
\begin{lem}\cite[Lemma 4.1]{Wa} \label{lem:phi}
For each $\underline{i}\in \mathfrak{P}(\mhcn)$, the map 
\begin{align}
\phi: P_{\underline{i}}\rightarrow \Lambda_{\underline{i}},\quad \tau\mapsto \tau\cdot \underline{i}=(i_{\tau^{-1}(1)}, i_{\tau^{-1}(2)},\ldots,i_{\tau^{-1}(n)})
\end{align}
is bijective. 
\end{lem}

For each $\underline{i}\in \mathfrak{P}(\mhcn)$, by Definition \ref{defn:cs-wt} we have $i_k\neq i_{k+1}$ for each $1\leq k\leq n-1$ and then by \eqref{eq:qi=qj} we have $\mathtt{q}(i_k)\neq \mathtt{q}(i_{k+1})$. Thus we can define two linear
operators $\Xi_k$ and $\Omega_k$ on the $\mpcn$-module $L(\underline{i})$ such that
for any $z\in L(\underline{i})$,
\begin{align}
\Xi_kz&:=-\Big(\frac{x_k+x_{k+1}}{x_k^2-x_{k+1}^2}+c_kc_{k+1}\frac{x_k-x_{k+1}}{x_k^2-x_{k+1}^2}\Big)z,\label{Deltak}\\
\Omega_kz&:=\Bigg(\sqrt{1-\frac{2(x_k^2+x_{k+1}^2)}{(x_k^2-x_{k+1}^2)^2}}\Bigg)z
=\Bigg(\sqrt{1-\frac{2(\mathtt{q}(i_k)+\mathtt{q}(i_{k+1}))}{(\mathtt{q}(i_k)-\mathtt{q}(i_{k+1}))^2}}\Bigg)z.\label{omegak}
\end{align}
Suppose $\underline{i}\in \mathfrak{P}(\mhcn)$. Recall the definition
of $L(\underline{i})^{\tau}$ from Remark~\ref{rem:Ltau} for
$\tau\in P_{\underline{i}}$. Denote by $V^{\underline{i}}$ the
$\mpcn$-module defined via
\begin{align} 
V^{\underline{i}}=\oplus_{\tau\in P_{\underline{i}}}
L(\underline{i})^{\tau}.\label{Dunderi}
\end{align}
Then we have the following due to \cite{Wa}.  
\begin{thm}\cite[Theorem 4.5]{Wa}\label{thm:Classficiation}
 Suppose $\underline{i}, \underline{j}\in \mathfrak{P}(\mhcn)$. Then,
 \begin{enumerate}
 \item $V^{\underline{i}}$ affords an irreducible $\mhcn$-module via
\begin{align}
s_kz^{\tau}= \left \{
 \begin{array}{ll}
 \Xi_kz^{\tau}
 +\Omega_kz^{s_k\tau},
 & \text{ if } s_k \text{ is admissible with respect to } \tau\cdot\underline{i}, \\
 \Xi_kz^{\tau}
 , & \text{ otherwise },
 \end{array}
 \right.\label{actionformula}
\end{align}
 for  $1\leq k\leq n-1, z\in L(\underline{i})$ and $\tau\in
P_{\underline{i}}$.  It has the same type as the irreducible
$\mpcn$-supermodule $L(\underline{i})$.

\item
$V^{\underline{i}}\cong V^{\underline{j}}$ if and only if
$\underline{i}\sim\underline{j}$.

\item Every irreducible completely splittable $\mhcn$-module in $\operatorname{Rep}_{\I}\mhcn$
is isomorphic to $V^{\underline{i}}$ for some $\underline{i}\in\mathfrak{P}(\mhcn)$. Hence the equivalence classes $\mathfrak{P}(\mhcn)/\sim$
parametrize irreducible completely splittable $\mhcn$-supermodules in the
category $\operatorname{Rep}_{\I}\mhcn$.

\end{enumerate}
\end{thm}

\subsection{Irreducible completely splittable representations of $\mathcal{Y}_n$}
Denote by $\mathcal{C}_n$ the subalgebra of
$\mhcn$ generated by $c_1,\ldots, c_n$, which is known as the
Clifford algebra. The Sergeev super algebra
$\mathcal{Y}_n=\mathcal{C}_n\rtimes\F \mathfrak{S}_n$ is isomorphic to the
subalgebra of $\mhcn$ generated by $c_1,\ldots, c_n, s_1,\ldots,
s_{n-1}$.  We first recall the classification of irreducible $\mathcal{Y}_n$-modules obtained in \cite{BK} (cf. \cite{K1}). 
A partition $\la$ is called \emph{$p$-strict} if $p$ divides $\la_r$ whenever $\la_r=\la_{r+1}$ for $r \geq 1$.  We say that a $p$-strict partition is \emph{$p$-restricted} if in addition
\[
\begin{cases}
\la_r-\la_{r+1}<p & \text{if }p\mid \la_r,\\
\la_r-\la_{r+1}\leq p & \text{if }p\nmid \la_r,
\end{cases}
\]
Denote by $\mathcal{RP}_p(n)$ be the set of $p$-restricted $p$-stricted partition of $n$. It is known that for each $\la\in\mathcal{RP}_p(n)$, there exists an irreducible $\mathcal{Y}_n$-supermodule $M^\la$. We refer the reader to \cite{K1} for details of the construction. Let
\[b(\la):=\sharp \{r\geq 1\mid p\nmid \la_r>0\}\]
be the number of (non-zero) parts of $\la$ that are not divisible by $p$.

\begin{thm}\cite[Theorem 22.2.1]{K1}\label{thm: Klesh}
The set $\{M^\la|\la\in \mathcal{RP}_p(n)\}$ forms a complete set of pairwise non-isomorphic irreducible $\mathcal{Y}_n$-module. Moreover, for $\la\in \mathcal{RP}_p(n)$, $M^\la$ is of type $\texttt{M}$ if $b(\la)$ is even, type $\texttt{Q}$ if $b(\la)$ is odd.
\end{thm}

The Jucys-Murphy elements $L_k(1\leq k\leq n)$ in
$\mathcal{Y}_n$ are defined as
\begin{align}
L_k=\sum_{1\leq j< k}(1+c_jc_k)(jk),\label{JM}
\end{align}
where $(jk)$ is the transposition exchanging $j$ and $k$ and keeping
all others fixed.

\begin{defn} A $\mathcal{Y}_n$-module is called
{\em completely splittable} if the Jucys-Murphy elements $L_k (1\leq
k\leq n)$ act semisimply.
\end{defn}

In the following, we shall identify those irreducible $\mathcal{Y}_n$-module $M^\la$ in Theorem \ref{thm: Klesh}  which are completely splittable by using the known classification Theorem \ref{thm:Classficiation} for the case of affine Sergeev superalgebra $\mhcn$. It is well known that there exists a surjective homomorphism
\begin{align}
\digamma: \mhcn&\rightarrow \mathcal{Y}_n\notag \\
c_k\mapsto c_k, s_l&\mapsto s_l, x_k\mapsto L_k, \quad (1\leq
k\leq n, 1\leq l\leq n-1)\label{eq:surj}
\end{align}
whose kernel $\op{Ker}\digamma=(x_1)$ coincides with the ideal of $\mhcn$ generated by $x_1$. 

\begin{rem}\label{rem:subcat}
By \eqref{eq:surj} the category of finite dimensional $\mathcal{Y}_n$-modules can
be identified as the category of finite dimensional $\mhcn$-modules
which are annihilated by $x_1$. By~\cite[Lemma 4.4]{BK} (cf.
\cite[Lemma 15.1.2]{K2}), a $\mhcn$-module $M$ belongs to the
category $\text{Rep}_{\I}\mhcn$ if all of eigenvalues of $x_j$ on
$M$ are of the form $\mathtt{q}(i)$ for some $1\leq j\leq n$. Hence the
category of finite dimensional completely splittable
$\mathcal{Y}_n$-module can be identified with the subcategory of
$\text{Rep}_{\I}\mhcn$ consisting of completely splittable
$\mhcn$-modules on which $x_1=0$. 
\end{rem}
 
Clearly by Remark \ref{rem:subcat} together with Theorem \ref{thm:Classficiation}, it is reduced to identify these $V^{\underline{i}}$ on which 
$x_1=0$ in order to obtain a classification of irreducible completely splittable $\mathcal{Y}_n$-modules. In fact, a nice combinatorial description has been provided in \cite[Section 6]{Wa} and we will review the details in the following. 

A partition $\la=(\la_1,\la_2,\ldots)$ with $\ell(\la)=\ell$ is said to be strict if $\la_1>\la_2>\cdots>\la_\ell>0$. Denote by $\mathcal{SP}(n)$ the set of strict partition of $n$. Similar to  the case of partitions, a strict partition $\la\in \mathcal{SP}(n)$ can be identified with the shifted Young diagram which is obtained from the ordinary Young diagram by shifting the $k$-th row to the right by $k-1$ squares for all $k>1$, that is, 
\begin{equation}
\la^{\mathsf{s}}=\{(i,j)\mid 1\leq i\leq\ell(\la), i\leq j\leq i+\la_i-1 \}. 
\end{equation} 
The $(i,j)$-hook of a shifted Young diagram contains all nodes that are either in the same row as $(i,j)$ and to the right of $(i,j)$, or in the same column as $(i,j)$ and below $(i,j)$ including $(i,j)$.  Additionally if $(j,j)$ is included in the hook then nodes in the $(j+1)$-row are also included. Denote by $h^{\mathsf{s}}_\la(i,j)$ the number of nodes in the $(i,j)$-hook.
\begin{example}
Suppose $\la=(7,5,3,2)$, then the $(1,2)$-hook and $(1,4)$-hook of shifted Young diagram $\la^{\mathsf{s}}$ are as below
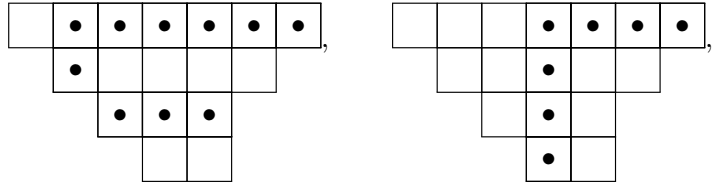
\[\ydiagram{7,1+5,2+3,3+2}*[\bullet]{1+6,1+1,2+3}, 
\qquad
\ydiagram{7,1+5,2+3,3+2}*[\bullet]{3+4,3+1,3+1,3+1},\]
and accordingly $h^{\mathsf{s}}_\la(1,2)=10$ and $h^{\mathsf{s}}_\la(1,4)=7$. 
\end{example}

Denote by $\mathcal{T}^{\mathsf{s}}(\la)$ the set of shifted tableaux of shape $\la^{\mathsf{s}}$; that is, a shifted tableau is a labelling of the nodes in the shifted Young diagram $\la^{\mathsf{s}}$ with the entries $1,2,\dots, n$. Let $T_{(i,j)}$ denote the entry in the node $(i,j)$ and let $T(k)$ be the node which is occupied by the number $k$ for each $1\leq k\leq n$. So if $T_{(i,j)}=k$ then $T(k)=(i,j)$. 
A shifted tableau $T$ is called \emph{standard} if its entries strictly increase from left to right along each row and down each column. We denote by $\op{Std}^{\mathsf{s}}(\la)$ the subset of $\mathcal{T}^{\mathsf{s}}(\la)$ consisting of standard tableaux of shape $\la^{\mathsf{s}}$. We then have the following remarkable hook length formula (cf. \cite[Chapter III, Section 8, Example 12]{Mac})
\begin{equation}\label{eq:hook2}
\sharp \op{Std}^{\mathsf{s}}(\la)=\frac{n!}{\prod_{(i,j)\in\la^{\mathsf{s}}} h^{\mathsf{s}}_\la(i,j)}
\end{equation}
where the product in the denominator is over all nodes in the shifted Young diagram $\la^{\mathsf{s}}$.

Following \cite[Lemma 6.6-6.7]{Wa}, we set 
\begin{align}
\mathcal{CP}^{\mathsf{s}}_p(n)=\Big\{\xi=(\xi_1,\xi_2,\ldots)\in\mathcal{SP}(n)\mid~ & \xi_1=p-u, \xi_2\leq u\text{ for some }1\leq u\leq\frac{p-3}{2} \notag\\
&\text{ or }1\leq \xi_1\leq\frac{p+1}{2}\Big\}\label{eq:cpn}
\end{align}
and in addition define
\begin{align}\label{eq:p-std}
\op{Std}^{\mathsf{s}}_p(\xi)=
\left\{
\begin{array}{ll}
\big\{T\in\op{Std}^{\mathsf{s}}(\xi)\mid T_{(2,\xi_2+1)}>T_{(1,\xi_1)}\big\},&\text{ if } \xi_1=p-u,\xi_2=u \text{ for some }\\
&\quad 1\leq u\leq \frac{p-3}{2},\\
\op{Std}^{\mathsf{s}}(\xi),&\text{ otherwise}. 
\end{array}
\right.
\end{align}
for  each $\xi\in\mathcal{CP}^{\mathsf{s}}_p(n)$. 
A tableau $T$ is called \emph{$p$-standard} if $T\in\op{Std}^{\mathsf{s}}_p(\xi)$ for some $\xi\in\mathcal{CP}^{\mathsf{s}}_p(n)$.  
We label the residue of nodes in the shifted Young diagram of $\la\in\mathcal{CP}^{\mathsf{s}}_p(n)$ 
using the set $\mathbb{I}$ in \eqref{defn:I} via  the way that the first node in each row has residue $0$ and then follow the repeating pattern
\begin{equation}\label{eq:residue}
0, 1,\ldots,\frac{p-3}{2},\frac{p-1}{2},\frac{p-3}{2},\ldots,1,0. 
\end{equation}
The residue  in \eqref{eq:residue} is actually to compute
$\mathtt{q}$-values of the usual residue $j-i$ of the nodes $(i,j)$ and the reason for this pattern is due to the observation $\mathtt{q}(a)=\mathtt{q}(b)$ if $a=b\mod p$ or $a+b+1=0\mod p$ for any $a,b\in\mathbb{Z}$. Let $\xi\in\mathcal{CP}^{\mathsf{s}}_p(n)$ and suppose $T\in \op{Std}^{\mathsf{s}}_p(\xi)$. Let 
$$
\underline{i}_T=(\op{res}(T(1)),\op{res}(T(2)),\cdots,\op{res}(T(n)))\in\mathbb{I}^n
$$ be the residue sequence corresponding to $T$. 

\begin{example}
Let $p=7$, so $\frac{p-1}{2}=3$. Then $\xi=(5,2,1)$ belongs to $\mathcal{CP}^{\mathsf{s}}_p(8)$. The residues of nodes in the shifted Young diagram $\xi^{\mathsf{s}}$ are as follows:
\[
\ydiagram{5,1+2,2+1}*[0]{1,1+1,2+1}*[1]{1+1,2+1}*[2]{2+1}*[3]{3+1}*[2]{4+1}. 
\]
In addition for 
$$
T=\ydiagram{5,1+2,2+1}*[1]{1}*[2]{1+1}*[3]{2+1}*[4]{3+1}*[6]{4+1}*[5]{0,1+1}*[7]{0,2+1}*[8]{0,0,2+1}\in\op{Std}_p^{\mathsf{s}}(\xi), 
$$
we have $\underline{i}_T=(0,1,2,3,0,2,1,0)$. 
\end{example}
\begin{lem}\cite[Lemma 6.6-6.7]{Wa}\label{lem:equiv}
Let $\xi\in\mathcal{CP}^{\mathsf{s}}_p(n)$ and suppose $T\in \op{Std}^{\mathsf{s}}_p(\xi)$. Then 
\begin{enumerate}
\item
$
\underline{i}_T\in\mathfrak{P}(\mhcn)
$ and moreover if $\underline{j}\in\mathfrak{P}(\mhcn)$, then $\underline{j}\sim \underline{i}_T$ if and only if $\underline{j}=\underline{i}_S$ for some $S\in\op{Std}^{\mathsf{s}}_p(\xi)$. Hence 
$\Lambda_{\underline{i}_T}=\{\underline{i}_S\mid S\in\op{Std}^{\mathsf{s}}_p(\xi)\}$. 

\item If $\xi\neq\gamma\in\mathcal{CP}^{\mathsf{s}}_p(n)$, then $\underline{i}_T\nsim\underline{i}_S$ for any $T\in\op{Std}_p^{\mathsf{s}}(\xi)$ and $S\in\op{Std}_p^{\mathsf{s}}(\gamma)$. 

\end{enumerate}
\end{lem}
Assume $\xi\in\mathcal{CP}^{\mathsf{s}}_p(n)$ and $T\in \op{Std}^{\mathsf{s}}_p(\xi)$. Recall the module $V^{\underline{i}}$ for each $i\in\mathfrak{P}(\mhcn)$ defined in \eqref{Dunderi} and set 
\begin{equation}\label{eq:VD-iso}
V^\xi:=V^{\underline{i}_T}. 
\end{equation}
Then $V^\xi$ admits a $\mhcn$-module by Theorem \ref{thm:Classficiation}. 
Clearly by Lemma \ref{lem:equiv} and Theorem \ref{thm:Classficiation} the $\mhcn$-module $V^\xi$ up to isomorphism is independent of the choice of $T\in \op{Std}^{\mathsf{s}}_p(\xi)$.

\begin{thm}(cf. \cite[Theorem 6.8]{Wa})
The set $\{V^{\xi}\mid \xi\in\mathcal{CP}^{\mathsf{s}}_p(n)\}$ is a complete set of non-isomorphic irreducible completely splittable $\mathcal{Y}_n$-modules. Moreover 
$$
\dim V^{\xi}=2^{n-\lfloor\frac{\ell(\xi)}{2}\rfloor}\sharp\op{Std}^{\mathsf{s}}_p(\xi). 
$$
\end{thm}
\begin{proof}
Fix $T\in\mathcal{CP}_p^{\mathsf{s}}(n)$. 
 Observe that for each $\tau\in P_{\underline{i}_T}$, there exists a unique $S\in\op{Std}^{\mathsf{s}}_p(\xi)$ such that $\tau\cdot \underline{i}_T=\underline{i}_S\in \Lambda_{\underline{i}_T}$ by Lemma \ref{lem:phi}  and Lemma \ref{lem:equiv}. This together with \eqref{eq:Ltau} gives rise to an isomorphism 
\begin{align}\label{eq:psi}
\psi^\tau_{T}: L(\underline{i}_T)^\tau\xrightarrow{\cong} L(\underline{i}_S)
\end{align}
Then by \eqref{eq:P1-irrep} and Lemma \ref{lem:irrepPn} we obtain $x_1=0$ on $L(\underline{i}_T)^\tau$ and hence $x_1=0$ on $V^\xi$. This means $V^\xi$ is actually a $\mathcal{Y}_n$-module.  Then by Theorem \ref{thm:Classficiation} and Lemma  \ref{lem:equiv}, one can obtain that $V^\xi\cong V^{\gamma}$ if and only if $\xi=\gamma\in\mathcal{CP}^{\mathsf{s}}_p(n)$.  
Then the theorem follows from \cite[Theorem 6.8]{Wa} by comparing the parametrizing set. 
\end{proof}
 Clearly for each $n\geq 1$, 
$
\mathcal{CP}^{\mathsf{s}}_p(n)\subseteq\mathcal{RP}_p(n)
$ and by \cite[Remark 6.9]{Wa}  we have 
\begin{equation}\label{eq:VM-iso}
V^\xi\cong M^\xi
\end{equation} for each $\xi\in \mathcal{CP}^{\mathsf{s}}_p(n)$.

\subsection{Irreducible representations of $\mathcal{Y}_n$ in the case $n=p$.}
 

\begin{lem}\label{p1}
For $n \geq 1$, the following holds:
\begin{align}
\mathcal{RP}_p(n) =
\begin{cases}
\mathcal{SP}(n), & \text{if } p > n; \\
\mathcal{SP}(n) \setminus \{(n)\}, & \text{if } p = n.
\end{cases}
\end{align}
\end{lem}

\begin{proof}
Suppose $p > n$. For any $\lambda \in \mathcal{RP}_p(n)$, if $\lambda_r = \lambda_{r+1}$ for some $r$, then $p \mid \lambda_r$. Since $\lambda_r \leq n < p$, it follows that $\lambda_r = 0$, which implies $\lambda \in \mathcal{SP}(n)$. Thus, $\mathcal{RP}_p(n) \subseteq \mathcal{SP}(n)$. Conversely, any $\mu \in \mathcal{SP}(n)$ is clearly $p$-strict since no non-zero parts are equal. Moreover, $\mu$ is $p$-restricted because $\mu_r - \mu_{r+1} \leq \mu_1 \leq n < p$ for all $r$. Thus, $\mu \in \mathcal{RP}_p(n)$, showing that $\mathcal{SP}(n) \subseteq \mathcal{RP}_p(n)$. Therefore, $\mathcal{RP}_p(n) = \mathcal{SP}(n)$ when $n < p$.

Now assume $n = p$. Note that the partition $(n)$ is not $p$-restricted since $\lambda_1 - \lambda_2 = n - 0 = p$, and $p \mid p$. Consequently, any $\lambda \in \mathcal{RP}_p(n)$ must satisfy $\lambda_1 \leq n - 1$. If $\lambda_r = \lambda_{r+1}$, then $p \mid \lambda_r$; since $\lambda_r \leq \lambda_1 < p$, we must have $\lambda_r = 0$, implying $\lambda \in \mathcal{SP}(n)$. Conversely, for any $\mu \in \mathcal{SP}(n)$ such that $\mu \neq (n)$, we have $\mu_1 \leq n - 1$. It follows that $\mu$ is $p$-strict and $\mu_r - \mu_{r+1} \leq \mu_1 \leq n - 1 < p$ for all $r$. Thus, $\mu \in \mathcal{RP}_p(n)$, which yields $\mathcal{RP}_p(n) = \mathcal{SP}(n) \setminus \{(n)\}$.
\end{proof}

\begin{lem}\label{p2}
For $n \geq 1$, the following holds:
\begin{align}
\mathcal{CP}^{\mathsf{s}}_p(n) =
\begin{cases}
\mathcal{SP}(n), & \text{if } p > n; \\
\mathcal{SP}(n) \setminus \{(n)\}, & \text{if } p = n.
\end{cases}
\end{align}
\end{lem}

\begin{proof}
Suppose $p > n$. By \eqref{eq:cpn}, it suffices to show that every $\lambda \in \mathcal{SP}(n)$ is contained in $\mathcal{CP}^{\mathsf{s}}_p(n)$. This clearly holds if $\lambda_1 \leq \frac{p+1}{2}$. Otherwise, suppose $\frac{p+3}{2} \leq \lambda_1 \leq n \leq p-1$ and let $u = p - \lambda_1$. Then $1 \leq u \leq \frac{p-3}{2}$. Since $\lambda \in \mathcal{SP}(n)$, we have $\lambda_2 \leq n - \lambda_1 \leq p - 1 - \lambda_1 < p - \lambda_1 = u$. Thus, $\lambda$ satisfies the conditions for $\mathcal{CP}^{\mathsf{s}}_p(n)$.

Now assume $n = p$. Clearly $(n) \notin \mathcal{CP}^{\mathsf{s}}_p(n)$ because $n > p-1$. For any $\lambda \in \mathcal{SP}(n)$ with $\lambda \neq (n)$, we have $\lambda_1 \leq n - 1 = p - 1$. As in the previous case, either $\lambda_1 \leq \frac{p+1}{2}$ or the condition $1 \leq u = p - \lambda_1 \leq \frac{p-3}{2}$ is satisfied. In the latter case, $\lambda_2 \leq n - \lambda_1 = p - \lambda_1 = u$, which implies $\lambda \in \mathcal{CP}^{\mathsf{s}}_p(n)$. Therefore, $\mathcal{CP}^{\mathsf{s}}_p(n) = \mathcal{SP}(n) \setminus \{(n)\}$.
\end{proof}

\begin{prop}\label{prop:RP-CP}
$\mathcal{RP}_p(n)=\mathcal{CP}^{\mathsf{s}}_p(n)$ if and only if $n\leq p$. Hence, all irreducible $\mathcal{Y}_n$-modules are completely splittable if and only if $n\leq p$.
\end{prop}
\begin{proof}
By Lemma \ref{p1} and Lemma \ref{p2}, we have $\mathcal{RP}_p(n)=\mathcal{CP}^{\mathsf{s}}_p(n)$ if $n\leq p$. Now suppose $p< n,$ then $n=ap+b$ for some $a\in \mathbb{Z}_+,$ $0\leq b \leq p-1$. Obviously either $\la=(p^{a},b)$(the case $b\neq 0$) or $\mu=(p^{a-1},p-1,1)$(the case $b=0$ and then $a> 1$) belongs to $\mathcal{RP}_p(n)$. But neither of them belongs to $\mathcal{CP}^{\mathsf{s}}_p(n)$ since $\la_1=\mu_1=p>p-1$. Hence $\mathcal{RP}_p(n)\neq \mathcal{CP}^{\mathsf{s}}_p(n)$. This proves the proposition. 
\end{proof}

\begin{prop}\label{prop:p-standard2}
Suppose $n\leq p$ and $\xi\in\mathcal{CP}^{\mathsf{s}}_p(n)$. Then
\[\sharp \op{Std}^{\mathsf{s}}_p(\xi)=\begin{cases}
\frac{n-2u+1}{n-u}\binom{n-2}{u-1},&\text{ if } \xi=(p-u,u)\text{ for some }1\leq u\leq \frac{p-3}{2},\\
\frac{n!}{\prod h^s_\xi(i,j)}, &\text{ otherwise}.
\end{cases}\]
\end{prop}
\begin{proof}
Clearly in the case $n<p$ we have $\xi\notin\{(p-u,u,\ldots)\in\mathcal{CP}_p^{\mathsf{s}}(n)\mid 1\leq u\leq \frac{p-3}{2}\}$ and  then the proposition follows from \eqref{eq:p-std} and \eqref{eq:hook2}.  It remains to consider the case $n=p$.  
For any $1\leq u\leq\frac{p-3}{2}$, the only partition satisfying $\xi_1=p-u$ and $\xi_2=u$ is $\xi=(n-u,u)=(p-u,u)$. Then by \eqref{eq:p-std} we have
\begin{equation}\label{eq:Stdp-equal-3}
\op{Std}^{\mathsf{s}}_p(\xi)=\op{Std}^{\mathsf{s}}(\xi)\text{ if }\xi \notin\{(p-u,u,\ldots)\in\mathcal{CP}_p^{\mathsf{s}}(n)\mid 1\leq u\leq \frac{p-3}{2}\}
\end{equation}
and then the proposition follows from \eqref{eq:hook2}. 
Now suppose $\xi=(n-u,u)=(p-u,u)$ with $1\leq u\leq\frac{p-3}{2}$. By \eqref{eq:p-std}, a tableau $T\in\op{Std}^{\mathsf{s}}(\xi)$ belongs to $\op{Std}^{\mathsf{s}}_p(\xi)$ if and only if $T_{(2,u+1)}>T_{(1,n-u)}$, which implies $T_{(2,u+1)}=n$. This leads to 
$
\op{Std}^{\mathsf{s}}_p(\xi)=\{T\in \op{Std}^{\mathsf{s}}(\xi)\mid T_{(2,u+1)}=n\}
$
and hence 
\begin{equation}\label{eq:Stdp-equal-4}
\sharp\op{Std}^{\mathsf{s}}_p(\xi)=\sharp\op{Std}^{\mathsf{s}}(\xi^-),
\end{equation}
where  $\xi^-=(n-u,u-1)\in\mathcal{SP}(n-1)$ due to $n-u>u-1$ as $1\leq u\leq \frac{p-3}{2}$. Now for $\xi^-=(n-u,u-1)$, 
it is straightforward to verify that the hook lengths $h^{\mathsf{s}}_{\xi^-}(i,j)$ of $(i,j)\in (\xi^-)^{\mathsf{s}}$ are $n-1, n-u,n-u-1,\ldots, n-2u+2, n-2u,n-2u-1,\ldots,1, u-1,\ldots, u$ and hence by \eqref{eq:hook2} we obtain 
$$
\sharp\op{Std}^{\mathsf{s}}(\xi^-)=\frac{(n-2u+1)(n-1)!}{(n-1)\cdot (n-u)! (u-1)!}=\frac{n-2u+1}{n-u}\binom{n-2}{u-1}. 
$$
Then the proposition follows from \eqref{eq:Stdp-equal-3}, \eqref{eq:Stdp-equal-4} and \eqref{eq:hook2}. 
\end{proof}
If $\xi\in\mathcal{CP}^{\mathsf{s}}_p(n)$ and $T\in\op{Std}^{\mathsf{s}}_p(\xi)$, we write 
$$
\kappa_T(k)=\mathtt{q}(\text{res}(T(k)))=\text{res}(T(k))(\text{res}(T(k))+1). 
$$
By  \eqref{lem:equiv} and Definition \ref{defn:cs-wt}, we obtain $\text{res}(T(k))\neq\text{res}(T(k+1))$ for $1\leq k\leq n-1$ and hence $\kappa_T(k)\neq \kappa_T(k+1)$. Thus we can introduce the following two well-defined elements in $\mpcn$: 
\begin{equation}\label{eq:XiT}
\Xi_k^T=-\left(\frac{1}{\sqrt{\kappa_T(k)}-\sqrt{\kappa_T(k+1)}}+\frac{c_kc_{k+1}}{\sqrt{\kappa_T(k)}+\sqrt{\kappa_T(k+1)}}\right)
\end{equation}
and 
\begin{equation}\label{eq:OmT}
\Omega_k^T=\sqrt{1-\frac{2(\kappa_T(k)+\kappa_T({k+1}))}{(\kappa_T(k)-\kappa_T({k+1}))^2}}
\end{equation}
for $1\leq k\leq n-1$.  
Then we are ready to introduce our main result.
\begin{thm}
Suppose $n=p$. Let $\xi\in\mathcal{RP}_p(n)=\mathcal{CP}^{\mathsf{s}}_p(n)$. Then 
\\
(1) If $ \xi\notin\{ (p-u,u)\mid 1\leq u\leq \frac{p-3}{2}\}$, then there exists a  $\mathcal{Y}_n$-module isomorphism 
\begin{equation*}
M^\xi\cong V^\xi\cong \oplus_{T\in\op{Std}^{\mathsf{s}}(\xi)}L(\underline{i}_T)
\end{equation*}
with the action of $\mathcal{C}_n\subset\mpcn$ on $L(\underline{i}_T)$ given via Lemma \ref{lem:irrepPn} and the action of $s_k\in\mathfrak{S}_n$ for $1\leq k\leq n-1$ is given by 
\begin{equation}\label{eq:sk-action-2}
s_k z=\left\{
\begin{array}{ll}
\Xi_k^T z+\Omega_k^Tz^{s_k},&\text{ if }s_kT\in\op{Std}^{\mathsf{s}}(\xi),\\
&\\
\Xi_k^T z,&\text{ otherwise}, 
\end{array}
\right.
\end{equation}
for each $z\in L(\underline{i}_T)$ and $T\in\op{Std}^{\mathsf{s}}(\xi)$. Moreover 
\begin{equation}\label{eq:dim-M-1}
\dim M^\xi=\frac{2^{n-\lfloor \frac{\ell(\xi)}{2}\rfloor}n!}{\prod_{(i,j)\in\xi^{\mathsf{s}}} h^{\mathsf{s}}_\xi(i,j)}. 
\end{equation}
\\
(2) If $\xi=(p-u,u)$ for some $1\leq u\leq\frac{p-3}{2}$, then there exists a  $\mathcal{Y}_n$-module isomorphism 
satisfies
\begin{equation*}
M^\xi\cong V^\xi\cong \oplus_{T\in\op{Std}^{\mathsf{s}}_p(\xi)}L(\underline{i}_T)
\end{equation*}
with the action of $\mathcal{C}_n\subset\mpcn$ on $L(\underline{i}_T)$ given via Lemma \ref{lem:irrepPn} and the action of $s_k\in\mathfrak{S}_n$ for $1\leq k\leq n-1$ is given by 
\begin{equation}\label{eq:sk-action-3}
s_k z=\left\{
\begin{array}{ll}
\Xi_k^T z+\Omega_k^Tz^{s_k},&\text{ if }1\leq k\leq n-2, s_kT\in\op{Std}^{\mathsf{s}}_p(\xi),\\
&\\
\Xi_k^T z,&\text{ if }1\leq k\leq n-2, s_kT\notin\op{Std}^{\mathsf{s}}_p(\xi),\\
&\\
\Xi_{n-1}^T z, &\text{ if }k=n-1. 
\end{array}
\right.
\end{equation}
for each $z\in L(\underline{i}_T)$ and $T\in\op{Std}^{\mathsf{s}}_p(\xi)$, where $\Xi_{n-1}^T$ satisifies
\begin{equation}\label{eq:sk-action-4}
\Xi^T_{n-1}=\left\{
\begin{array}{ll}
(\frac{1}{\sqrt{u(u+1)}-\sqrt{u(u-1)}}+\frac{c_{n-1}c_n}{\sqrt{u(u+1)}+\sqrt{u(u-1)}}),&\text{ if } T_{(1,n-u)}=n-1 ,\\
&\\
(\frac{1}{\sqrt{(u-1)(u-2)}-\sqrt{u(u-1)}}+\frac{c_{n-1}c_n}{\sqrt{(u-1)(u-2)}+\sqrt{u(u-1)}}),&\text{ if }T_{(2,u)}=n-1. 
\end{array}
\right.
\end{equation}
 Moroever 
\begin{equation}\label{eq:dim-M-2}
\dim M^\xi=\frac{2^{n-\lfloor \frac{\ell(\xi)}{2}\rfloor}(n-2u+1)}{n-u}\binom{n-2}{u-1}. 
\end{equation}
In addition, the set $\{V^\xi\mid \xi\in\mathcal{CP}_p^{\mathsf{s}}(n)=\mathcal{SP}(n)\setminus(n)\}$ is a complete set of non-isomorphic irreducible $\mathcal{Y}_n$-modules. 
\end{thm}
\begin{proof}
Firstly, by Proposition \ref{prop:RP-CP} and \eqref{eq:VM-iso} we have $\mathcal{RP}_p(n)=\mathcal{CP}_p^{\mathsf{s}}(n)=\mathcal{SP}(n)\setminus(n)$ and $M^\xi\cong V^\xi$ for each $\xi\in\mathcal{RP}_p(n)$. 
 Observe that each permutation $\tau$ naturally acts on an arbitrary tableau of shifted Young diagram $\xi^{\mathsf{s}}$ to a get a new tableau $\tau\cdot S$ by permuting the entries in $S$ and moreover 
 \begin{equation}\label{eq:tau-act1}
\tau\cdot \underline{i}_S=\underline{i}_{\tau\cdot S}. 
 \end{equation}
In particular, fixing $T^\xi\in \op{Std}^{\mathsf{s}}_p(\xi)$,  one can obtain $\underline{i}_{\tau\cdot T^\xi}\in\Lambda_{\underline{i}_{T^\xi}}$ for each $\tau\in P_{\underline{i}_{T^\xi}}$ by Lemma \ref{lem:phi} and then by Lemma \ref{lem:equiv}
we have $\tau\cdot T^\xi\in \op{Std}^{\mathsf{s}}_p(\xi)$. This leads to 
$$
\op{Std}^{\mathsf{s}}_p(\xi)=\{\tau\cdot T^\xi\mid \tau\in P_{\underline{i}_{T^\xi}}\}. 
$$
and moreover 
 \begin{equation}\label{eq:P-equiv}
P_{\underline{i}_{T^\xi}}=\{\tau=s_{k_t}\cdots s_{k_2}s_{k_1}~\mid~ s_{k_a}\cdots s_{k_2}s_{k_1}\cdot T^\xi\in\op{Std}^{\mathsf{s}}_p(\xi) \text{ for }1\leq a\leq t\}
 \end{equation}
This together with \eqref{eq:tau-act1} and Remark \ref{rem:Ltau} leads to 
$$
V^{\underline{i}_{T^\xi}}=\oplus_{\tau\in P_{\underline{i}_{T^\xi}}}L(\underline{i}_{T^\xi})^\tau\cong \oplus_{T\in \op{Std}^{\mathsf{s}}_p(\xi)}L(\underline{i}_T)
$$
Moreover, for each $z\in L(\underline{i}_{T^\xi})$ and $\tau\in P_{\underline{i}_{T^\xi}}$, we have $z^\tau\in L(\underline{i}_{T^\xi})^\tau\cong L(\underline{i}_T)$ with $T=\tau\cdot T^\xi$ by \eqref{eq:tau-act1} and moreover if $s_k$ is admissible with respect to $\tau\cdot \underline{i}_{T^\xi}$ then 
\begin{equation}\label{eq:z-sktau}
z^{s_k\tau}=(z^{\tau})^{s_k}
\end{equation} by using the twist action in Remark \ref{rem:Ltau}. Then by the isomorphism $\psi^\tau_{T^\xi}: L(\underline{i}_{T^\xi})^\tau\rightarrow L(\underline{i}_T)$ in \eqref{eq:psi} with $T=\tau\cdot T^\xi$, we can identify $z^\tau$ with its image in $L(\underline{i}_T)$ and then the action of $s_k$ in \eqref{eq:sk-action-2} follows from \eqref{actionformula} in Theorem \ref{thm:Classficiation}. In addition, the dimensional formula \eqref{eq:dim-M-1} is due to Lemma \ref{lem:irrepPn} and Proposition \ref{prop:p-standard2}. This together \eqref{eq:VD-iso} and \eqref{eq:VM-iso}  with proves part (1). 

One can prove the action formula \eqref{eq:sk-action-3} and the dimension formula \eqref{eq:dim-M-2} in part (2) applying the same approach with some extra computation need to verify \eqref{eq:sk-action-4} for the action of $s_{n-1}$. Actually, if $\xi=(p-u,u)$ for some $1\leq u\leq \frac{p-3}{2}$, then  $T_{(2,u+1)}=n$ since any $T\in\op{Std}^{\mathsf{s}}_p(\xi)$ satisfies $T_{(2,u+1)}>T_{(1,p-u)}$ by \eqref{eq:p-std}. This means $s_{n-1}T\notin\op{Std}^{\mathsf{s}}_p(\xi)$. Then by  \eqref{actionformula} in Theorem \ref{thm:Classficiation} we obtain 
$$
s_{n-1}z=\Xi_{n-1}^Tz
$$
for each $z\in L(\underline{i}_T)$ and $T\in\op{Std}^{\mathsf{s}}_p(\xi)$. Since $T_{(2,u+1)}=n$, one can deduce that either  $T_{(2,u)}=n-1$ with $u\geq 2$ or $T_{(1,p-u)}=n-1$. This means $\text{res}(T(n))=u-1\in\I$ and $\text{res}(T(n-1))=u-2\in \I$ with $u\geq 2$ or  $\text{res}(T(n-1))=u\in \I$ according to the residue labeling pattern defined in \eqref{eq:residue}. Then by \eqref{eq:XiT} the equation \eqref{eq:sk-action-4} is verified. This proves the theorem. 
\end{proof}
}

\end{document}